\documentclass[draft]{article}
\def\today{10.1.13} 
\usepackage{amsmath,amsfonts,amsthm,amssymb,amscd}
\usepackage{graphicx,epsf,amsmath}  
\theoremstyle{plain} \newtheorem{theorem}{Theorem}[section]
\newtheorem{lemma}[theorem]{Lemma}

 \theoremstyle{definition}
\newtheorem{definition}[theorem]{Definition} \theoremstyle{remark}
\newtheorem{remark}[theorem]{Remark}

\newcommand{\R}{{\mathbb R}} \newcommand{\U}{{\mathcal U}}

\newcommand{\Z}{{\mathbb Z}}

\newcommand{\Tr}{{\mathcal T}_K}

\newcommand{\resto}{{\mathcal
R}}

\newcommand{\Sc}{{\mathcal S}}

\newcommand{\T}{\mathbb{T}}

\def\norma#1{\left\| #1\right\|}

\def\sleq{\leq\kern-6pt \cdot\null\hskip4pt}

\def\norma#1{\left\| #1\right\|}

\def\sleq{\preceq}

\newtheorem{thm}{Theorem}[section]

\def\U{{\mathcal U}}

\title{Almost global existence for a fractional
  Schrodinger equation on spheres and tori}
\author{Dario Bambusi, Yannick Sire}

\date{\today}

\begin{document}
\maketitle

\begin{abstract}
We study the time of existence of the solutions of the following
Schr\"odinger equation
$$i\psi_t = (-\Delta)^s \psi +f(|\psi|^2)\psi,\,\,\, x \in \mathbb
S^d\ ,\ or\ x\in\T^d$$ where $(-\Delta)^s$ stands for the spectrally
defined fractional Laplacian with $s>1/2$ and $f$ a smooth function. We prove an almost global
existence result for almost all $s>1/2$.
\end{abstract}

\section{Introduction and result}
We consider the equation 

\begin{equation}\label{schro}
i\psi_t = (-\Delta)^s \psi +f(|\psi|^2)\psi,\,\,\, x \in \mathbb S^d
\ or\ x\in\T^d\ ,
\end{equation}
where $(-\Delta)^s$ is the spectrally defined fractional Laplacian,
i.e. the $s$-th power of the Laplace-Beltrami operator $-\Delta$ and
$f$ a function of class $C^{\infty}$ in a neighborhood of 0. 

We will study such an equation in high regularity Sobolev spaces. We
denote by $H^r$ the Sobolev space of the $L^2$ functions which admit
$r$ weak derivatives which are square integrable. We will endow it with
the norm
\begin{equation}
\label{norm}
\norma{\psi}^2_{H^r}:=\int\left[\left|\psi(x)\right|^2+\bar\psi(x)(-\Delta
  )^r\psi(x)\right] d x\ .
\end{equation}
Our main result is the following  theorem.

\begin{theorem}
\label{main}
Fix $K\geq 1$, then there exists a set $\Sc\subset(1/2,+\infty)$, having
zero measure, such that, for any $s\in (1/2,+\infty)-\Sc$ there exists a
positive $r_K$ and for any $r>r_K$ there exist $\epsilon_{r,K}, T_{r,K}$
with the following property: if the initial datum $\psi_0$ fulfills 
\begin{equation}
\label{ini}
\epsilon:=\norma{\psi_0}_{H^r}<\epsilon_{r,K}\ ,
\end{equation}
then one has
\begin{equation}
\label{stima}
\norma{\psi(t)}_{H^r}<2\epsilon\ ,\quad \forall
\left|t\right|\leq\frac{T_{r,K}}{\epsilon^K}\ .
\end{equation}
\end{theorem}

We remark that our method is unable to deal with $s\in(0,1/2)$, due to
the fact that the growth of the frequencies is too slow and thus the
differences between couple of frequencies can accumulate on open sets.

The theorem is an application of Theorem 4.3 (and Proposition 4.1) of
\cite{bam08}, which in turn is based on the Birkhoff normal form
theory developed in \cite{Bam01,BG, BDGS,gre07,bam08}. 
Some care is
needed in order to apply such a theory, because in our case one of the
frequencies vanishes and thus, a priori, the normal form does not
allow to control the motion of the corresponding mode. Furthermore,
one has to check that the frequencies fulfill the strong nonresonance
condition introduced in those papers. The first problem is solved
exploiting the Gauge invariance of the equation (more or less as in
\cite{bamsacc}), the second one by a variant of the method used in
\cite{Bam01} for the wave equation. 

We recall that such a theory has been recently extended by Delort to
some quasilinear equations \cite{delort,delHal}, however we did not
investigate the applicability of his method to the present case.

Concerning existence of the dynamics of the fractional Schr\"odinger
equation, local existence of smooth solution is trivial. On the
contrary, as far as the dimension is larger then 1, very little is
known on the time of existence of solutions, indeed, to the knowledge of the
authors, the only existing results are those of \cite{GW} in $\R^d$, where
dispersion is exploited in order to prove global-wellposedness and
scattering for small data. On the contrary nothing is known on compact
manifold, where such a mechanism clearly fails.

\section{Proof of Theorem \ref{main}}\label{proof}

To be determined we focus on the case of the sphere which is
slightly more difficult. The case of tori is almost
identical. 

We expand $\psi $ in spherical harmonics
$$\psi=\sum_{j,k}\xi_{jk}Y_{jk}(x)$$
and 
$$\overline \psi=\sum_{j,k}\eta_{jk}Y_{jk}(x)$$
where $Y_{jk}$ are the spherical harmonics.

The Hamiltonian writes

\begin{equation}\label{hamil}
H(\xi, \eta)=\sum_{ j \geq 0} \omega_j \sum_{k} \xi_{jk} \eta_{jk} + H_p (\xi,\eta)
\end{equation}
where $\omega_j=(j(j+d-1))^s$ and $H_p$ has a zero of order $3$ at the
origin. Then eq. \eqref{schro} is equivalent to the Hamilton equations 
\begin{equation}
\label{hamilt}
\dot \xi_{jk}=- i\frac{\partial H}{\partial \eta_{jk}}\ ,\quad \dot
\eta_{jk}= i\frac{\partial H}{\partial \xi_{jk}}\ .
\end{equation}
We will also use the Poisson brackets of two functions $F,G$ on the
phase space, which are defined by 
\begin{equation}
\label{poi}
\left\{F;G\right\}:=i\sum_{jk} \left[\frac{\partial F}{\partial
    \eta_{jk}}\frac{\partial G}{\partial
    \xi_{jk}}-\frac{\partial F}{\partial
    \xi_{jk}} \frac{\partial G}{\partial
    \eta_{jk}}\right]\ .
\end{equation}

The theory developed in \cite{ds,Bam01,BG, BDGS,gre07,bam08} applies
to Hamiltonians of the form \eqref{hamil} in which the nonlinear part
belongs to a suitable class (called functions with localized
coefficients in \cite{bam08}). It was proved in \cite{bam08} that in
our case the nonlinearity belongs to such a class, so we refer to that
paper for the proof and the precise definition of such a property.

We come to the nonresonance property of the frequencies. 

\begin{definition}
\label{Knr}
Fix $K\geq 3$, then the frequencies $(\omega_1,...,\omega_{\infty})$
are said to fulfill the property ($K$-NR) if there exist $\gamma>0,$
and $\alpha\in\R$ such that for any $N$ large enough one has
\begin{eqnarray}
\label{nr.1}
\left|\sum_{j\geq 1}\omega_jL_j
\right|\geq\frac{\gamma}{N^\alpha}\ ,
\end{eqnarray}
for any $L\in\Z^\infty$, fulfilling
$0\not=|L|:=\sum_j|L_j|\leq K+2$, $\sum_{j>N}|L_j|\leq2$.
\end{definition}

We are going to prove that for almost all $s$ in the considered
interval such a property is fulfilled.

\begin{thm}\label{freqTh}
There exists a zero measure set $\mathcal S \subset (1/2,+\infty)$ 
such that if $s \in  (1/2,+\infty)-\mathcal S$ then the frequencies 
$$(\omega_1,..., \omega_\infty, ...)$$
fulfill the property $(K-NR)$
\end{thm}

The proof of such a theorem is a straightforward generalization of the
proof of Theorem 4.4 of \cite{bam08}. The only difference is given by
the estimate from below of the determinant of the matrix formed by the
vectors of the derivatives of the frequencies with respect to
$s$. Such an estimate is done explicitly in the following Lemma
\ref{det}.

We also remark that the condition $s>1/2$ is needed in order to
pass from Lemma 6.8 of \cite{bam08} to Lemma 6.9 of the same paper,
namely for passing from the estimate of linear combinations of
frequencies with index smaller then a fixed $K$ to linear combinations
involving also a couple of arbitrary large indexes (as in the second
Melnikov condition of KAM theory).

\begin{lemma}\label{det}
For any $\kappa \leq K$, consider $\kappa $ indexes $1 \leq j_1\leq \cdot \leq j_\kappa \leq K$; consider the matrix $D$ given by 

$$\begin{pmatrix} \omega_{j_1} & \cdot \cdot \cdot& \omega_{j_\kappa}
  \\ d \omega_{j_1}/ds & \cdot \cdot \cdot& d\omega_{j_\kappa}/ds
  \\ \vdots & \vdots &\vdots 
\\ d^{\kappa-1}
  \omega_{j_1}/ds^{\kappa-1} & \cdot \cdot \cdot&
  d^{\kappa-1}\omega_{j_\kappa}/ds^{\kappa-1}
\end{pmatrix}
$$ Denote by $\mathcal D$ its determinant. Then there exists $C>0$
s.t. the following estimate holds
$$|\mathcal D| \geq \frac{C}{K^{\kappa^2}}.$$
\end{lemma}

\begin{proof}
Denote 
$$\lambda_j = j (j+d-1)$$
then one has 
$$\frac{d^k \omega_j}{ds^k}=(\ln \lambda_j)^k \omega_j\ .
$$
Therefore 
$$\mathcal D= \omega_{j_1} \cdot \cdot \cdot \omega_{j_\kappa}
\left|\begin{matrix}
1& \cdot  \cdot  \cdot& 1 \\
x_{j_1}& \cdot  \cdot  \cdot& x_{j_\kappa}\\
  \vdots  &  \vdots &\vdots \\
x_{j_1}^{\kappa-1} & \cdot  \cdot  \cdot& x_{j_\kappa}^{\kappa-1}Ê
\end{matrix}\right|
=\omega_{j_1} \cdot \cdot \cdot \omega_{j_\kappa}
\prod_{1 \leq l <k \leq \kappa
}(x_{j_k}-x_{j_l})=\omega_{j_1} \cdot \cdot \cdot \omega_{j_\kappa}
\prod_{1 \leq l <k \leq \kappa
}\ln \frac{\lambda_{j_k}}{\lambda_{j_l}}$$ where $x_j:=\ln \lambda_j$. To fix ideas take
$j_k>j_l$, and let $\delta=\lambda_{j_k}-\lambda_{j_l}=(j_k-j_l)(j_k+j_l
+d-1)>d.$ Then

$$\ln \frac{\lambda_{j_k}}{\lambda_{j_l}}= \ln (1+\frac{\delta}{\lambda_{j_l}})\geq \ln (1+\frac{d}{\lambda_{j_l}}) \geq \frac{C}{j_l^2}.$$Thus 
$$\prod_{1 \leq \ell <k \leq \kappa }\ln \frac{\lambda_{j_k}}{\lambda_{j_l}} \geq \prod \frac{C}{K^2} \geq \frac{C}{K^{2\kappa^2}}.$$

Since $\omega_j \geq 1$ for all $j$ the thesis follows. 
\end{proof}

Thus, proceeding as in \cite{bam08} one gets that the frequencies
fulfill the nonresonance condition and Theorem 4.15 of \cite{bam08}
applies. Such a theorem ensures that, defining 
\begin{equation}
\label{actions}
I_j:=\sum_l \xi_{jl}\eta_{jl}\equiv\sum_{l}\left|\xi_{jl}\right|^2\ ,
\end{equation}
the following holds.

\begin{theorem}
\label{main1}
(Theorem 4.15 of \cite{bam08})Fix $K\geq1$, then there exists a finite
$r_K$ a neighborhood $\U_{r_K}^{(K)}$ of the origin in $H^{r_K}$ and a
canonical transformation $\Tr:\U_{r_K}^{(K)}\to H^{r_K}$ which puts
the system in normal form up to order $K+3$, namely s.t.
\begin{equation}
\label{eq:bir1}
H^{(K)}:=H\circ \Tr=H_0+Z^{(K)}+\resto^{(K)}
\end{equation}
where $Z^{(K)}$ and $\resto^{(K)}$ have smooth vector field
\begin{itemize}
\item[(i)]$Z^{(K)}$ is a polynomial of degree $K+2$ which Poisson
  commutes with $I_j$ for all $j\not=0$ (but not necessarily with
  $I_0$;
\item[(ii)] $\resto^{(K)}$ has a small vector field, i.e.
\begin{equation}
\label{resto1}
\norma{X_{\resto^{(K)}}(\psi)}_{H^{r_K}}\leq C\norma{\psi}_{H^{r_K}}^{K+2}\ ,\quad
\forall \psi\in\U_{r_K}^{(K)}\ ;
\end{equation}
\item[(iii)] one has
\begin{equation}
\label{def1}
\norma{\psi-\Tr(\psi)}_{H^{r_K}} \leq C\norma{\psi}_{H^{r_K}}^2\ ,\quad \forall
\psi\in\U_{r_K}^{(K)}\ .
\end{equation}
An inequality identical to \eqref{def1} is fulfilled by the inverse
transformation $\Tr^{-1}$.
\item[(iv)] For any $r\geq r_K$ there exists a subset
  $\U_r^{(K)}\subset\U_{r_K}^{(K)}$ open in $H^r$ such that the
  restriction of the canonical transformation to $\U^{(K)}_r$ is
  analytic also as a map from $H^r\to H^r$ and the inequalities
  \eqref{resto1} and \eqref{def1} hold with $r$ in place of $r_K$.
\end{itemize}\end{theorem}

This theorem however is not enough to control the solution since
as emphasized at point (i), due to the fact that the zero mode has
zero frequency, $I_0\equiv \xi_0\eta_0$ does not commute with
$Z^{(K)}$ and thus its modulus can a priori grow in an unbounded
way. However, as we are going to show in a while, this cannot happen
due to the Gauge invariance.

\begin{remark}
Due to Gauge invariance, the $L^2$ norm is preserved for the original
nonlinear dynamics and since $-\Delta$ is self-adjoint on $L^2$
one has
$$\Gamma(\xi,\eta):=\int_{\mathbb S^d} \psi \overline \psi \,dx=
\sum_{jk} \eta_{jk}\xi_{jk}= \sum_{j,k}|\xi_{jk}|^2. $$
\end{remark}

\begin{remark}\label{gauge}
Expanding $H_p$ is Taylor series one has
$$H_p(\xi, \eta)= \sum_{J,L} H_{JL} \eta^L \xi^J\ ,\quad
\eta^L:=\prod_{jk}\eta_{jk}^{L_{jk}}$$ and similarly for $\xi^J$.  Due
to Gauge invariance $H_{JL} \neq 0$ implies
$$\left \{ \Gamma, \xi^L\eta^J \right \}=i\sum_{jk} (L_{jk}-J_{jk})
\xi^L \eta^J =0 \ ,$$
which, in turn implies 
\begin{equation}
\label{G.1}
\sum_{jk} (L_{jk}-J_{jk})=0
\end{equation}
\end{remark}

\noindent
{\it End of the proof of Theorem \ref{main}.}  The Gauge invariance is
conserved after the Birkhoff normal form transformation (see
e.g. \cite{bamsacc}). It follows that {\bf any monomial $\xi^L\eta^J$
  which is present in the normal form $Z^{(K)}$ fulfills the property
  \eqref{G.1}.} However, if a monomial is present in the normal form
it must also commute with all the $I_j$, $j\not=0$. It follows that
the indexes must fulfill
\begin{equation}
\label{G.2}
\sum_{l}(L_{jk}-J_{jk})=0\ ,\quad \forall j\not=0\ ,
\end{equation}
which together with \eqref{G.1} implies $L_0-J_0=0$, which in turn
implies that $\left\{Z^{(K)},I_0\right\}=0$. Then Theorem \ref{main}
follows exactly in the same way in which Proposition 4.1 of
\cite{bam08} follows from Theorem 4.3 of that paper. \qed


\begin{thebibliography}{BDGS07}

\bibitem[Bam03]{Bam01}
Dario Bambusi.
\newblock Birkhoff normal form for some nonlinear {PDE}s.
\newblock {\em Comm. Math. Phys.}, 234(2):253--285, 2003.

\bibitem[Bam08]{bam08}
D.~Bambusi.
\newblock A {B}irkhoff normal form theorem for some semilinear {PDE}s.
\newblock In {\em Hamiltonian dynamical systems and applications}, NATO Sci.
  Peace Secur. Ser. B Phys. Biophys., pages 213--247. Springer, Dordrecht,
  2008.

\bibitem[BDGS07]{BDGS}
D.~Bambusi, J.-M. Delort, B.~Gr{\'e}bert, and J.~Szeftel.
\newblock Almost global existence for {H}amiltonian semilinear {K}lein-{G}ordon
  equations with small {C}auchy data on {Z}oll manifolds.
\newblock {\em Comm. Pure Appl. Math.}, 60(11):1665--1690, 2007.

\bibitem[BG06]{BG}
D.~Bambusi and B.~Gr{\'e}bert.
\newblock Birkhoff normal form for partial differential equations with tame
  modulus.
\newblock {\em Duke Math. J.}, 135(3):507--567, 2006.

\bibitem[BS07]{bamsacc}
Dario Bambusi and Andrea Sacchetti.
\newblock Exponential times in the one-dimensional {G}ross-{P}itaevskii
  equation with multiple well potential.
\newblock {\em Comm. Math. Phys.}, 275(1):1--36, 2007.

\bibitem[Del11]{delHal}
Jean-Marc Delort.
 {Quasi-linear perturbations of Hamiltonian Klein-Gordon equations on
  spheres}.
 November 2011.\hfill\null  {\tt http://hal.archives-ouvertes.fr/hal-00643474/PDF/article.pdf}

\bibitem[Del12]{delort}
J.-M. Delort.
\newblock A quasi-linear {B}irkhoff normal forms method. {A}pplication to the
  quasi-linear {K}lein-{G}ordon equation on {$\Bbb S^1$}.
\newblock {\em Ast\'erisque}, (341):vi+113, 2012.

\bibitem[DS04]{ds}
J.-M. Delort and J.~Szeftel.
\newblock Long-time existence for small data nonlinear {K}lein-{G}ordon
  equations on tori and spheres.
\newblock {\em Int. Math. Res. Not.}, (37):1897--1966, 2004.

\bibitem[Gr{\'e}07]{gre07}
Beno{\^{\i}}t Gr{\'e}bert.
\newblock Birkhoff normal form and {H}amiltonian {PDE}s.
\newblock In {\em Partial differential equations and applications}, volume~15
  of {\em S\'emin. Congr.}, pages 1--46. Soc. Math. France, Paris, 2007.

\bibitem[GW]{GW}
Z.~Guo and Y.~Wang.
\newblock Improved strichartz estimates for a class of dispersive equations in
  the radial case and their applications to nonlinear schrdinger and wave
  equation.
\newblock {\em Preprint}.

\end{thebibliography}

\def\cprime{$'$} \def\cprime{$'$} \def\cprime{$'$}

\end{document}